\documentclass[12pt]{article}

\usepackage[margin=1in]{geometry}
\usepackage{enumerate}
\usepackage{amsmath}
\usepackage{amssymb,latexsym}
\usepackage{amsthm}
\usepackage{color}
\usepackage{graphicx}
\usepackage{hyperref}
\usepackage{amscd}
\usepackage{cancel}
\usepackage{lineno,xcolor}
\DeclareMathOperator{\Tr}{Tr}

\DeclareMathOperator{\C}{\mathcal{C}}

\title{Evasive subspaces \vspace{.3cm}}

\usepackage{authblk}

\author[$\dagger$]{Daniele Bartoli}
\author[$\star$]{Bence Csajb\'ok }
\author[$\ddag$]{Giuseppe Marino}
\author[$\ddag$]{Rocco Trombetti}

\affil[$\dagger$]{Department of Mathematics and Informatics, University of Perugia, Perugia,  Italy, \textit{daniele.bartoli@unipg.it}\vspace*{.3cm}}

\affil[$\star$]{ELKH--ELTE Geometric and Algebraic Combinatorics Research Group,
	ELTE E\"otv\"os Lor\'and University, Budapest, Hungary, 
	Department of Geometry,
	1117 Budapest, P\'azm\'any P.\ stny.\ 1/C, Hungary,	\textit{csajbokb@cs.elte.hu}\vspace*{.3cm}}

\affil[$\ddag$]{Dipartimento di Matematica e Applicazioni ``R. Caccioppoli'',
	Universit\`a degli Studi di Napoli ``Federico II'',
	Via Cintia, Monte S.Angelo I-80126 Napoli, Italy, \textit{giuseppe.marino@unina.it, rtrombet@unina.it}}

\date{}

\newcommand{\cC}{{\mathcal C}}

\newcommand{\cG}{{\mathcal G}}

\newcommand{\cF}{{\mathcal F}}

\newcommand{\cD}{{\mathcal D}}

\newcommand{\cS}{{\mathcal S}}

\newcommand{\F}{{\mathbb F}}

\newcommand{\la}{\langle}
\newcommand{\ra}{\rangle}

\newcommand{\V}{{\mathbb V}}

\newtheorem{theorem}{Theorem}[section]

\newtheorem{corollary}[theorem]{Corollary}
\newtheorem{definition}[theorem]{Definition}
\newtheorem{proposition}[theorem]{Proposition}
\newtheorem{result}[theorem]{Result}
\newtheorem{example}[theorem]{Example}

\newtheorem{remark}[theorem]{Remark}

\DeclareMathOperator{\PG}{{PG}}

\begin{document}
	\maketitle
	
	\begin{abstract}
		Let $V$ denote an $r$-dimensional vector space over $\F_{q^n}$, the finite field of $q^n$ elements. 
		Then $V$ is also an $rn$-dimension vector space over $\F_q$. An $\F_q$-subspace $U$ of $V$ is $(h,k)_q$-evasive if it meets the $h$-dimensional $\F_{q^n}$-subspaces of $V$ in $\F_q$-subspaces of dimension at most $k$. 
		The $(1,1)_q$-evasive subspaces are known as scattered and they have been intensively studied in finite geometry, their maximum size 
		has been proved to be $\lfloor rn/2 \rfloor$ when $rn$ is even or $n=3$. 
		
		We investigate the maximum size of $(h,k)_q$-evasive subspaces, study two duality relations among them and provide various constructions. 
		In particular, we present the first examples, for infinitely many values of $q$, of maximum scattered subspaces when $r=3$ and $n=5$. 
		We obtain these examples in characteristics $2$, $3$ and $5$.

		%We study additive subspace evasive sets.
	\end{abstract}
	
	%\bigskip
	%{\it AMS subject classification:}
	

	\section{Introduction}
	
	Let $\cF$ be a set of subsets of a set $A$ and let $S \subseteq A$. 
	In \cite[Definition 1]{PR2004} Pudl\'ak and R\"odl called $S$ \emph{$c$-evasive for $\cF$} if for all $W\in \cF$
	\[ | W \cap S|\leq c.\]
	In their work $A$ was then taken to be the set of vectors of an $r$-dimensional vector space $V$ over $\F_2$, the finite field of two elements, and $\cF$ was the set of all $d$-dimensional affine subspaces of $V$ for some positive integer $d$. 
	The authors of \cite{PR2004} also showed how such evasive sets can be used to construct explicit Ramsey graphs. 
	Later, these objects were called \emph{subspace evasive}, see Guruswami \cite[Section 4]{G2011}, \cite[Definition 2]{GWX} or Dvir and Lovett \cite[Definition 3.1]{DL2012} and were studied intensively since they can be used to obtain explicit list decodable codes with optimal rate and constant list-size. More precisely, following Guruswami, let $V$ be an $r$-dimensional vector space over the (usually but not necessarily finite) field $\F$. Then $S \subseteq V$ is called \emph{$(d,c)$-subspace evasive} if for every $d$-dimensional linear subspace $H$ of $V$ we have $|S \cap H|\leq c$. 
	
	Note that Dvir and Lovett used the term $(d,c)$-subspace evasive to denote $c$-evasive sets for the set of all $d$-dimensional affine subspaces. The two concepts do not coincide, however, they are strongly related as the next two paragraphs show.
	
	Every $d$-dimensional affine subspace is contained in a $(d+1)$-dimensional linear subspace. Thus if $S$ is $c$-evasive for the set of all $(d+1)$-dimensional linear subspaces then it is $c$-evasive for the set of all $d$-dimensional affine subspaces.
	
	Also, if $\F$ is a finite field, say $\F_{p^s}$, for some $p$ prime, and $S$ is additive (or equivalently, $S$ is an $\F_p$-linear subspace of $V$), then $S$ is $c$-evasive for the set of all $d$-dimensional linear subspaces if and only if it is $c$-evasive for the set of all $d$-dimensional affine subspaces. To see this, consider any $d$-dimensional affine subspace $A$ and a vector $x\in S \cap A$. 
	For a set $B$ we denote by $B-x$ the difference set $\{b-x : b \in B\}$. Then
	\[|S \cap A|=|(S-x) \cap (A-x)|=|S \cap (A-x)|,\]
	where $S-x=S$ follows from the additivity of $S$, and $A-x$ is a $d$-dimensional linear subspace of $V$. 
	
	In \cite{GWX} additive subspace evasive sets were constructed, that is, 
	subspace evasive sets which are also linear subspaces of $V$ over some subfield $\F_{p^t}$ of $\F_{p^s}$, $t\mid s$. Clearly, if $\F_q$ is a subfield of $\F_{p^s}$, then $p^s=q^n$ for some integer $n$. The aim of this paper is to study these evasive sets.
	Note that, if $S$ is linear over $\F_q$ then intersections of $S$ with linear (or affine) $\F_{q^n}$-subspaces of $V$ are also linear (or affine) $\F_q$-subspaces. From now on, we will denote by $V=V(r,q^n)$ an $r$-dimensional vector space over the finite field $\F_{q^n}$. Note that $V$ is also an $rn$-dimensional vector space over $\F_q$.
	
	\begin{definition}
	\label{maindef}
	An $\F_q$-subspace $U$ of $V$ will be called  \emph{$(h,k)_q$-evasive} if $\la U \ra_{\F_{q^n}}$ has dimension at least $h$ over $\F_{q^n}$ and the $h$-dimensional $\F_{q^n}$-subspaces of $V$ meet $U$ in $\F_q$-subspaces of dimension at most $k$.
	\end{definition}
	
	Note that in the definition above the condition on the dimension of $\la U \ra_{\F_q^n}$ is to exclude trivial examples for which some of our results would not apply. 
	So, take an $\F_q$-subspace $U$ of $V$ such that the condition $\dim_{q^n}\la U \ra_{q^n}\geq h$ holds. 
	Then it is easy to find an $h$-dimensional $\F_{q^n}$-subspace meeting $U$ in an $\F_q$-subspace of dimension at least $h$ and hence for an $(h,k)_q$-evasive subspace $h\leq k$ must hold. Clearly, if $\dim_q U \leq k$ then $U$ is an $(h,k)_q$-evasive subspace.
	The $(r,k)_q$-evasive subspaces are the $\F_q$-subspaces of dimension at most $k$ which span $V$ over $\F_{q^n}$. 
	Note that an $(h,k)_q$-evasive subspace is also $(h',k')_q$-evasive for any $h'\leq h$ and $k'\geq k$.
	
	\medskip
	If $S$ is $c$-evasive for $\cF$ then the same holds for every subset $S' \subseteq S$. Thus there are two natural questions to ask:
	\begin{enumerate}
		\item[(A)] For given $c$ and $\cF$, what is the size of the largest $c$-evasive set for $\cF$? We will call evasive sets of this size \emph{maximum}.
		\item[(B)] For given $c$ and $\cF$, determine the smallest $c$-evasive sets for $\cF$ which are not contained in a larger one. We will call these evasive sets \emph{maximal}.
	\end{enumerate}
	
	For example if $\cF$ is the set of edges of a graph $\cG$ and $c=1$ then (A) asks for the size of a maximum independent set in $\cG$ and (B) asks for the size of a minimum vertex cover in $\cG$.
	
	\medskip
	
	The concept of evasive sets is well known in finite geometry as well. 
	Denote by $\Sigma$ a finite projective space isomorphic to $\PG(d,q)$. 
	A \emph{cap of kind $h$} is an $h$-evasive set for the set of $(h-1)$-dimensional projective subspaces of $\Sigma$, cf. \cite{Tallini1961}.
	The most studied examples are the \emph{arcs} ($h=d$), \emph{caps} ($h=2$) \cite{Hirsbook} and \emph{tracks} ($h=d-1$) \cite{DeBoer}.
	To arcs and tracks correspond the MDS and almost MDS codes, respectively (\cite{Simeonbook}, \cite{LR2015}).
	One can weaken further these conditions and consider point sets meeting each hyperplane in at most $n$ points.
	For example \emph{$(k,n)$-arcs} are the $n$-evasive point sets of size $k$ for the set of lines in a projective plane of order $q$.
	There are many famous conjectures regarding the size of a maximum evasive set in this setting.
	For example the maximal arc conjecture, which was proved by Ball, Blokhuis and Mazzocca \cite{BBM, BB}. 
	The main conjecture for MDS codes is equivalent to ask the maximum size of an arc in $\Sigma$ \cite[Section 7]{Simeonbook}. 
	
	\medskip
	
	Recently, in \cite{THR} Randrianarisoa introduced $q$-systems. 	A \emph{$q$-system} $U$ over $\F_{q^n}$ with parameters $[m,r,d]$ is an $m$-dimensional $\F_q$-subspace generating over $\F_{q^n}$ a $r$-dimensional $\F_{q^n}$-vector space $V$, where 
	\[d=m-\max\{\dim(U \cap H) : H \mbox{ is a hyperplane of $V$}\}.\]
	With our notation, it is equivalent to say that $U$ is $(r-1,m-d)_q$-evasive in $V(r,q^n)$ and it is not $(r-1,m-d+1)_q$-evasive. 
These objects are in one-to-one correspondence with $\F_{q^n}$-linear $[m,r,d]$-rank metric codes, cf. {\cite[Theorem 2]{THR}}.
	
	\medskip
	
	In \cite{CsMPZ2019} the authors investigated the following subspace analogue of caps of kind $h$: for $0<h<r$ an $\F_q$-subspace $U$ of $V=V(r,q^n)$ is called \emph{$h$-scattered} if $U$ generates $V$ over $\F_{q^n}$ and any $h$-dimensional $\F_{q^n}$-subspace of $V$ meets $U$ in an $\F_q$-subspace of dimension at most $h$. With the notation of this paper, the $h$-scattered subspaces are the $(h,h)_q$-evasive subspaces generating $V$ over $\F_{q^n}$. 
	
	\medskip
	
	A $t$-spread of $V$ is a partition of $V\setminus \{{\bf 0}\}$ by $\F_q$-subspaces of dimension $t$. In particular 
	$\cD:=\{ \la {\bf u} \ra_{\F_{q^n}}\setminus \{{\bf 0}\} : {\bf u} \in V\}$ is the so called Desarguesian $n$-spread of $V$. 
	An $\F_q$-subspace $U$ of $V$ is called \emph{scattered with respect to a spread $\cS$} if $U$ meets each element of $\cS$ in at most a one-dimensional $\F_q$-subspace, i.e. when $U$ is $q$-evasive for $\cS$. In \cite{BL2000} Blokhuis and Lavrauw proved that $rn/2$ is the maximum dimension of a scattered subspace of $V$ w.r.t. a Desarguesian $n$-spread. In other words, the dimension of a maximum $(1,1)_q$-evasive subspace is at most $rn/2$.
	After a series of papers it is now known that this bound is sharp when $rn$ is even, cf. Result \ref{result}. 
	Note that the $1$-scattered subspaces are the scattered subspaces generating $V$ over $\F_{q^n}$.
	
	In \cite{CsMPZ2019} the authors generalized the Blokhuis--Lavrauw bound and proved that the dimension of an $h$-scattered subspace is at most $rn/(h+1)$. 
	They also introduced a relation, called \emph{Delsarte duality}, on $\F_q$-subspaces of $V$ and proved that the Delsarte dual of an $rn/(h+1)$-dimensional $h$-scattered subspace is $h'$-scattered with dimension $r'n/(h'+1)$ in some vector space $V'(r',q^n)$. For the precise statement see \cite[Theorem 3.3]{CsMPZ2019}. Delsarte duality is a well known concept in the theory of rank metric codes. There is a correspondence between maximum $(r-1)$-scattered subspaces, also called \emph{scattered subspaces w.r.t. hyperplanes}, and certain MRD (maximum rank metric) codes, see \cite{Lu2017, ShVdV}. In fact, if $\cC$ is the MRD-code correspoinding to an $n$-dimensional $(r-1)$-scattered subspace $U$ then the Delsarte dual $\cC^{\perp}$ of $\cC$ is the MRD-code  corresponding to the Delsarte dual $\bar U$ of $U$ which is again maximum scattered w.r.t. hyperplanes; see \cite[Remark 4.11]{CsMPZ2019}.
	
	\medskip
	
	In Section \ref{Sec:2} we collect existing constructions and present some new ones to obtain large $(h,k)_q$-evasive subspaces. 
	In particular we study the direct sum of an $(h_1,k_1)_q$-evasive and an $(h_2,k_2)_q$-evasive subspace under various conditions. 
	In Section \ref{DD} we investigate the ``ordinary'' duality (induced by a non-degenerate reflexive sesquilinear form) and Delsarte duality on 
	evasive subspaces. In Section \ref{maximumsize} we present (in some cases sharp) upper bounds on the size of maximum evasive subspaces. 
	In this direction the first open problem is to determine the size of a maximum scattered (i.e. $(1,1)_q$-evasive) subspace in $V(3,q^5)$. From the Blokhuis--Lavrauw bound it follows that scattered subspaces of $V(3,q^5)$ have dimension at most $\lfloor (3\cdot5)/2 \rfloor=7$. 
	In Section \ref{Sec:5} we present scattered subspaces with this dimension for infinite many values of $q$. 
	We obtain these examples in characteristics $2$, $3$ and $5$. Note that these are the first 
	non-trivial examples (i.e. $n>3$) for scattered subspaces of dimension $\lfloor rn/2 \rfloor$ in $V(r,q^n)$ when $rn$ is odd. 
	To obtain these examples we use MAGMA and combine $q$-subresultants \cite{csajb} with the strategy of Ball, Blokhuis and Lavrauw from \cite{BBL2000} where the authors construct maximum scattered subspaces of dimension $6$ in $V(3,q^4)$. Thanks to the dualities described in Section \ref{DD}, our examples yield maximum evasive subspaces for many other parameters as well.
	
	\section{General properties and existence results}
	\label{Sec:2}

	In this section we prove basic properties of evasive subspaces and present some construction methods. 
	Some of these results are inductive and depend on already existing constructions as an input.
	We start by presenting the simplest construction of scattered subspaces w.r.t. hyperplanes. They correspond to MRD-codes known as Gabidulin codes, see \cite[Section 3.2]{SheekeySurvey}. 
	
	\begin{example}
		\label{Gabidulin}
		If $n\geq r$, then the $n$-dimensional $\F_q$-subspace
		\[\{(x,x^q,\ldots,x^{q^{r-1}}) : x \in \F_{q^n}\}\]
		is maximum $(r-1,r-1)_q$-evasive in $\F_{q^n}^r$.
	\end{example}
	
	The following example defines subgeometries $\PG(r,q^m)$ in $\PG(r,q^n)$.
	
	\begin{example}
		\label{subgeom}
If $m \mid n$ then the $mr$-dimensional $\F_q$-subspace
		\[\{(x_1,x_2,\ldots,x_r) : x_i \in \F_{q^m}\}\]
		is $(h,mh)_q$-evasive in $\F_{q^n}^r$ for each $h$. 
	\end{example}
	
	\medskip
	
	Next we collect what is known about maximum $(h,h)_q$-evasive subspaces. 
	
	\begin{result}
		\label{result}
		\begin{enumerate}
			\item If $h+1 \mid r$ and $n\geq h+1$, then maximum $(h,h)_q$-evasive subspaces of $V(r,q^n)$ have dimension $rn/(h+1)$, cf. \cite[Theorem 2.7]{CsMPZ2019};
			\item if $rn$ is even, then maximum $(1,1)_q$-evasive subspaces of $V(r,q^n)$ have dimension $rn/2$, cf. \cite{BBL2000, BGMP2018, BL2000, CsMPZ2017};
			\item if $rn$ is even, then maximum $(n-3,n-3)_q$-evasive subspaces of $V(r(n-2)/2,q^n)$ have dimension $rn/2$, cf. \cite[Corollary 3.5]{CsMPZ2019};
			\item in $V(r,q^n)$ the $h$-scattered subspaces of dimension $rn/(h+1)$ are $(r-1,rn/(h+1)-n+h)_q$-evasive, cf. \cite{BL2000} for $h=1$ and \cite{CsMPZ2019} for $h>1$.
		\end{enumerate}
	\end{result}
	
	Next recall Examples \ref{Guru1} and \ref{Guru2} of Guruswami \cite[Section B]{GWX} which are special (linear) cases of the construction of Dvir and Lovett \cite[Theorem 3.2]{DL2012}. 
	Suppose $n \geq r$ and take distinct elements $\gamma_1,\gamma_2,\ldots,\gamma_r \in \F_{q^n}^*$.
	For $1\leq i \leq h$ define
	\[f_i(x_1,x_2,\ldots,x_r)=\sum_{j=1}^{r} \gamma_j^i x_j^{q^{r-j}}.\]
	
	\begin{example}
		\label{Guru1}
		The $\F_q$-subspace 
		\[V_{\F_{q^n}}(f_1,f_2,\ldots,f_h):=\{(x_1,x_2,\ldots,x_r) : f_i(x_1,x_2,\ldots,x_r)=0\mbox{ for } i=1,\ldots,h\}\]
		is $(k,(r-1)k)_q$-evasive of dimension $n(r-h)$ for every $1\leq k\leq h<r$ in $\F_{q^n}^r$.
		The fact that $\la V_{\F_{q^n}}(f_1,f_2,\ldots,f_h) \ra_{\F_{q^n}}=\F_{q^n}^r$ is left as an exercise. 
		Note that when $r=2$ (and hence $h=k=1$) then $V_{\F_{q^n}}(f_1)$ is equivalent to Example \ref{Gabidulin}.
	\end{example}
	
	\begin{example}
		\label{Guru2}
		The direct sum of $s$ copies of $V_{\F_{q^n}}(f_1,f_2,\ldots,f_h)$ in $\F_{q^n}^r\oplus\ldots\oplus\F_{q^n}^r \cong V(rs,q^n)$ is a $(k,(r-1)k)_q$-evasive subspace of dimension $sn(r-h)$ for every $1\leq k\leq h<r$.
		
		Thus for every divisor $m$ of $r'$, in $V(r',q^n)$ there exist $(k,(m-1)k)_q$-evasive subspaces of dimension
		$nr'-hnr'/m$ for every $1\leq k \leq h <m$. 
		
		Note that when $m=2$ (and hence $h=k=1$) we obtain maximum scattered subspaces of $V(r,q^n)$, $r$ even, which are direct sum of the maximum scattered subspaces of $V(2,q^n)$ arising from Example \ref{Gabidulin}. It is equivalent to the construction of Lavrauw \cite[pg. 26]{LThesis}. 
	\end{example}
	
	\begin{proposition}
		\label{scendere}
		If $U$ is a $(h,k)_q$-evasive subspace in $V(r,q^n)$, then it is also $(h-s,k-s)_q$-evasive.
	\end{proposition}
	\begin{proof}
		It is enough to prove the result for $s=1$ and then apply induction.
		Suppose for the contrary that there exists an $(h-1)$-dimensional $\F_{q^n}$-subspace $H$  meeting $U$ in at least $q^{k}$ vectors.
		Since $\la U \ra_{\F_{q^n}}$ is not contained in $H$, there exists ${\bf u}\in U \setminus H$ and hence $\la {\bf u}, H\ra_{q^n}$ meets $U$ in at least $q^{k+1}$ vectors, a contradiction.
	\end{proof}
	
	\begin{proposition}
		\label{banale}
		If there exists an $(h,k)_q$-evasive subspace $U$ of dimension $t$ in $V(r,q^n)$, then
		there also exists an $(h,k+s)_q$-evasive subspace of dimension $t+s$ for each $0\leq s \leq rn-t$.
	\end{proposition}
	\begin{proof}
		Take ${\bf w}\notin U$. The $\F_q$-subspace $\la U,{\bf w}\ra_{\F_q}$ is $(h,k+1)_q$-evasive of dimension $t+1$. Indeed, suppose for the contrary that there exist ${\bf u_i} \in U$ such that
		\[{\bf w+u_1, w+u_2, \ldots, w+u_{k+2}},\] 
		are $\F_q$-linearly independent elements contained in the same $h$-dimensional $\F_{q^n}$-subspace $H$ of $V$. Then ${\bf u_1-u_{k+2}, u_2-u_{k+2}, \ldots, u_{k+1}-u_{k+2}}$ are $k+1$ $\F_q$-linearly independent elements of $U$ in $H$, contradicting the fact that $U$ is $(h,k)_q$-evasive. The result follows by induction. 
	\end{proof}

	In Proposition \ref{banale}, 
	starting from an evasive subspace of $V(r,q^n)$, we construct another evasive subspace in the same vector space. 
	In the next result we construct an evasive subspace by enlarging an evasive subspace lying in a hyperplane of $V(r,q^n)$.
	
	\begin{proposition}
		\label{onebyone}
		If there exists an $(r-1,k)_q$-evasive subspace of dimension $d$ in $V(r,q^n)$, then for a positive integer $s$ with $d-k \leq s \leq n$ in $V(r+1,q^n)$ there exists an $(r,k+s)_q$-evasive subspace of dimension $d+s$.
	\end{proposition}
	\begin{proof}
		Let $W$ be an $(r-1,k)_q$-evasive subspace of dimension $d$ in $V(r,q^n)$. Embed $V(r,q^n)$ as a hyperplane of $V(r+1,q^n)$ and take a vector ${\bf v}\notin V(r,q^n)$. Let $W'$ be an $\F_q$-subspace of $\la {\bf v}\ra_{q^n}$ of dimension $s$. The $\F_q$-subspace $W\oplus W'$ of $V(r+1,q^n)$ will be sufficient for our purposes.
	\end{proof}

	\begin{theorem}
		\label{dirsum}
		If $U_i$ is an $(h,k_i)_q$-evasive subspace of $V_i=V(r_i,q^n)$ for $i=1,2$, then $U=U_1 \oplus U_2$ is
		$(h,k_1+k_2-h)_q$-evasive  in $V=V_1 \oplus V_2$.
	\end{theorem}
	\begin{proof}
		Recall that $k_i\geq h$, for $i=1,2$, as we explained after Definition \ref{maindef}. By way of contradiction suppose that there exists an $h$-dimensional $\F_{q^n}$-subspace $W$ of $V$ such that
		\begin{equation}
			\label{eq1}
			\dim_{q}(W\cap U)\geq k_1+k_2-h+1.
		\end{equation}
		Clearly, $W$ cannot be contained in $V_i$ since $U_i$ is $(h,k_i)_q$-evasive in $V_i$
		and $k_1+k_2-h+1$ is larger than both $k_1$ and $k_2$.
		Let $W_1:=W\cap V_1$ and $s:=\dim_{q^n}W_1$. Then $s < h$ and by Proposition \ref{scendere}, the $\F_q$-subspace $U_1$ is $(s,k_1-h+s)_q$-evasive in $V_1$, thus 
		\begin{equation}
		\label{eq2}
			\dim_{q} (U_1 \cap W_1) \leq k_1-h+s.
		\end{equation}
		Denoting
		$\la U_1, W\cap U\ra_{\F_{q}}$ by $\bar U_1$, the Grassmann formula yields
		\[\dim_{q} \bar{U}_1-\dim_{q} U_1
		=\dim_{q}(W \cap U)-\dim_{q}(W \cap U_1)\]
		and hence by \eqref{eq1} and \eqref{eq2}
		\begin{equation}\label{eq4}
			\dim_{q} \bar{U}_1-\dim_{q} U_1
			\geq (k_1+k_2-h+1)-(k_1-h+s)=k_2+1-s.
		\end{equation}
		Consider the subspace $T:=W+V_1$ of the quotient space $V/V_1\cong V_2$. Then $\dim_{q^n}T=h-s$ and $T$ contains the $\F_q$-subspace $M:=\bar U_1+V_1$.
		Since $M$ is also contained in the $\F_q$-subspace $U+V_1=U_2+V_1$ of $V/V_1$, then $M$ is $(h,k_2)_q$-evasive in $V/V_1$ and hence also $(h-s,k_2-s)_q$-evasive. Hence $\dim_{q}(M \cap T)\leq k_2-s$.
		
		On the other hand,
		\[\dim_{q}(M\cap T)=\dim_{q} M=\dim_{q} \bar U_1-\dim_{q} (\bar U_1\cap V_1)\geq\]
		\[\dim_{q} \bar U_1 -\dim_{q}(U\cap V_1)=\dim_{q} \bar U_1 -\dim_{q} U_1,\]
		and hence, by \eqref{eq4},
		\[k_2-s \geq \dim_{q}(M\cap T)\geq k_2+1-s;\]
		a contradiction.
	\end{proof}
	
	Note that Example \ref{Guru2} is obtained from Example \ref{Guru1} by considering direct sum of $(h,(r-1)h)_q$-evasive subspaces so that their sum is evasive with the same parameters. By Theorem \ref{dirsum} we would get only an $(h,(r-1)h+(r-2)h)_q$-evasive subspace. The reason behind this is the  additional property of Example \ref{Guru1}, i.e. the fact that it is $(k,(r-1)k)_q$-evasive for each $1\leq k \leq h$. Assuming a similar hypothesis, and copying either the proof of Theorem \ref{dirsum} or the proof of \cite[Claim 3.4]{DL2012} one can prove the following.
	
	\begin{theorem}
		\label{dirsum1}
		If $U_i$ is a $(t,\lambda t)_q$-evasive subspace of $V_i=V(r_i,q^n)$ with $i=1,2$ for each $1\leq t\leq h$ and for some positive integer $\lambda$, then $U=U_1 \oplus U_2$ is $(t,\lambda t)_q$-evasive in $V=V_1 \oplus V_2$ for each $1\leq t\leq h$.
	\end{theorem}
	
	With $V=1$ then the result above states that the direct sum of two $t$-scattered subspaces is again $t$-scattered, see \cite[Theorem 2.5]{CsMPZ2019}. 
	When $h=1$ then Theorem \ref{dirsum1} is just a special case of the following result, which also improves on Theorem \ref{dirsum}.
	
	\begin{theorem}
		\label{directsumpoints}
		If $U_i$ is a $(1,k_i)_q$-evasive subspace of $V_i=V(r_i,q^n)$ for $i=1,2$, then $U=U_1 \oplus U_2$ is a
		$\left(1,\max\{k_1,k_2\}\right)_q$-evasive subspace of $V=V_1 \oplus V_2$.
	\end{theorem}
	\begin{proof}
		W.l.o.g. assume $k_2\geq k_1$ and put $\dim_{q} U_i=d_i$ for $i=1,2$.
		Suppose for the contrary that in $V$ there exists a one-dimensional $\F_{q^n}$-subspace $X$ such that $\dim_q (U \cap X)\geq k_2+1$.
		Clearly, $X \notin V_1 \cup V_2$.
		Consider the $(r_1+1)$-dimensional $\F_{q^n}$-subspace $T:=\la V_1, X \ra_{\F_{q^n}}$. 
		Then $Y:=T \cap V_2$  is a one-dimensional $\F_{q^n}$-subspace.
		Now consider the $\F_q$-subspace $T':=\la U_1, X \cap U\ra_{\F_q}$. By our assumption on $X$, $\dim_{q} T' \geq d_1+k_2+1$. Then
		\[\dim_q \la T', U_2\ra + \dim_{q} (T' \cap U_2) = \dim_{q} T' + \dim_{q} U_2.\]
		Since $\la T', U_2\ra_{\F_{q}}=U$, this yields 
		$\dim_q (T' \cap U_2) \geq k_2+1$.
		Also, from $T \geq T'$ and $V_2 \geq U_2$, we have $Y=T \cap V_2 \geq T' \cap U_2$ and hence the one-dimensional $\F_{q^n}$-subspace $Y$ meets $U_2$ in an $\F_q$-subspace of dimension at least $k_2+1$, a contradiction.
	\end{proof}
	
\section{Dualities on evasive subspaces}
\label{DD}
	
	\subsection{``Ordinary'' dual of evasive subspaces}
	
	Let $\sigma: V\times V\longrightarrow \F_{q^n}$ be a non-degenerate
	reflexive sesquilinear form on $V=V(r,q^n)$ and define
	\begin{equation}\label{form:traccia}
		\sigma' \colon ({\bf u}, {\bf v})\in V\times V \rightarrow
		\Tr_{q^n/q}(\sigma({\bf u}, {\bf v}))\in \F_q,
	\end{equation}
	where $\Tr_{q^n/q}$ denotes the trace function of $\F_{q^n}$ over $\F_q$.
	Then $\sigma'$ is a non-degenerate reflexive sesquilinear form on
	$V$, when $V$ is  regarded as an $rn$-dimensional vector space  over
	$\F_{q}$. Let $\tau$ and $\tau'$ be the orthogonal complement maps
	defined by $\sigma$ and $\sigma'$ on the lattices  of the
	$\F_{q^n}$-subspaces and $\F_{q}$-subspaces of $V$, respectively.
	Recall that  if $R$ is an
	$\F_{q^n}$-subspace of $V$ and $U$ is an $\F_{q}$-subspace of $V$
	then $U^{\tau'}$ is an $\F_q$-subspace of $V$, $\dim_{q^n}R^\tau+\dim_{q^n}R= r$ and
	$\dim_q U^{\tau'}+\dim_q U= rn$. It easy to see
	that $R^\tau=R^{\tau'}$ for each $\F_{q^n}$-subspace $R$ of $V$
	(for more details see \cite[Chapter 7]{Taylor}).
	
	Also, $U^{\tau '}$ is called the \emph{dual} of $U$ (w.r.t. $\tau'$). 
	Up to $\Gamma\mathrm{L}(r,q^n)$-equivalence, the dual of an $\F_q$-subspace of $V$ does not depend on the choice of the non-degenerate reflexive sesquilinear forms $\sigma$ and $\sigma'$ on $V$. For more details see \cite{Polverino}. If $R$ is an $s$-dimensional $\F_{q^n}$-subspace of $V$ and $U$ is a $t$-dimensional $\F_q$-subspace of $V$, then
	\begin{equation}\label{pesi}
		\dim_q(U^{\tau'}\cap R^\tau)-\dim_q(U\cap R)=rn-t-sn.
	\end{equation}
	From the previous equation the next result immediately follows.
	
	\begin{proposition}
		\label{d2nuovo}
		Let $U$ be a $t$-dimensional $(h,k)_q$-evasive subspace in $V =V(r,q^n)$, with $k<hn$. 
		Then $U^{\tau'}$ is an $(rn-t)$-dimensional $(r-h,(r-h)n+k-t)_q$-evasive subspace in $V$.
	\end{proposition}
	\begin{proof}
	From \eqref{pesi} it follows that the $(r-h)$-dimensional $\F_{q^n}$-subspaces meet $U^{\tau'}$ in subspaces of dimension at most $(r-h)n+k-t$. 
	The only thing to prove is the fact that $\la U^{\tau'} \ra_{\F_{q^n}}$ has dimension at least $r-h$. If this was not true then one could find an
	$(r-h-1)$-dimensional $\F_{q^n}$-subspace containing $U^{\tau'}$ and hence also an $(r-h)$-dimensional $\F_{q^n}$-subspace containing $U^{\tau'}$. 
	Such a subspace would meet $U^{\tau'}$ in an $\F_q$-subspace of dimension $rn-t$ which is larger than $(r-h)n+k-t$, a contradiction.
	\end{proof}
	
	Note that in Proposition \ref{d2nuovo} the condition $k<hn$ is not very restrictive since with $k\geq hn$ each $\F_q$-subspace of $V$ is $(h,k)_q$-evasive.
	
	\begin{corollary}
		\label{corord}
		Let $U$ be a $t$-dimensional $(h,q^h)$-evasive subspace in $V =V(r,q^n)$ with $t>h$. 
		If $\dim_{q^n} \la U^{\tau'}\ra_{\F_{q^n}} \geq r-h$, then $U^{\tau'}$ is maximum $(r-h,(r-h)n+h-t)_q$-evasive with dimension $rn-t$.
	\end{corollary}
	\begin{proof}
		Suppose for the contrary that there exists an $(rn-t+1)$-dimensional $(r-h,(r-h)n+h-t)_q$-evasive subspace $W$ in $V$.
		Then $\dim_q W^{\tau'}= t-1$ and $W^{\tau'}$ meets the $h$-dimensional $\F_{q^n}$-subspaces of $V$ in 
		$\F_q$-subspaces of dimension at most $h-1$. Thus $\dim_q W^{\tau'} \leq h-1$. It means that $t\leq h$, a contradiction.
	\end{proof}

	\subsection{Delsarte dual of evasive subspaces}
	
	In this section we follow \cite[Section 3]{CsMPZ2019}. 
	Let $U$ be a $t$-dimensional $\F_q$-subspace of a vector space $V=V(r,q^n)$, with $t>r$. By \cite[Theorems 1, 2]{LuPo2004} (see also \cite[Theorem 1]{LuPoPo2002}), there is an embedding of $V$ in $\V=V(t,q^n)$ with $\V=V \oplus \Gamma$ for some $(t-r)$-dimensional $\F_{q^n}$-subspace $\Gamma$ such that
	$U=\la W,\Gamma\ra_{\F_{q}}\cap V$, where $W$ is a $t$-dimensional $\F_q$-subspace of $\V$, $\langle W\rangle_{\F_{q^n}}=\V$ and $W\cap \Gamma=\{{\bf 0}\}$.
	Then the quotient space $\V/\Gamma$ is isomorphic to $V$ and under this isomorphism $U$ is the image of the $\F_q$-subspace $W+\Gamma$ of $\V /\Gamma$.

	\medskip
	
	Now, let $\beta'\colon W\times W\rightarrow\F_{q}$ be a non-degenerate reflexive sesquilinear form on $W$. Then $\beta'$ can be extended to a non-degenerate reflexive sesquilinear form $\beta\colon \V\times\V\rightarrow\F_{q^n}$.
	Let $\perp$ and $\perp'$ be the orthogonal complement maps defined by $\beta$ and $\beta'$ on the lattice of $\F_{q^n}$-subspaces of $\V$ and of $\F_q$-subspaces of $W$, respectively.
	For an $\F_q$-subspace $S$ of $W$ the $\F_{q^n}$-subspace $\la S \ra_{\F_{q^n}}$ of $\V$ will be denoted by $S^*$. In this case $(S^*)^{\perp}=(S^{\perp'})^*$.
	
	\medskip
	
	The next result and definition come from \cite[Proposition 3.1 and Definition 3.2]{CsMPZ2019}.
	
	\begin{proposition}
		\label{prop:dual2}
		Let $W$, $V$, $\Gamma$, $\V$, $\perp$ and $\perp'$ be defined as above. If $U$ is a $t$-dimensional $(h,t-r+h-1)_q$-evasive subspace of $V$ 
		with $t>r$, then $W+\Gamma^\perp$ is a subspace of $\V/ \Gamma^\perp$ of dimension at least $(t-r+h+1)$ (and at most $t$) over $\F_q$. If $h$ is maximal with this property then $W+\Gamma^\perp$ has dimension exactly $t-r+h+1$. 
	\end{proposition}
	\begin{proof}
As described above, $U$ turns out to be isomorphic to the $\F_q$-subspace $W+\Gamma$ of the quotient space $\V/\Gamma$.
Also, an $h$-dimensional $\F_{q^n}$-subspace of $\mathbb{V}/\Gamma$ corresponds to a $(t-r+h)$-dimensional $\F_{q^n}$-subspace of $\V$ containing $\Gamma$. 
Hence,  $\dim_q(H\cap W) \leq t-r+h-1$ for each $(t-r+h)$-dimensional subspace $H$ of $\V$ containing $\Gamma$. 
Next we prove that the dimension of $W\cap \Gamma^\perp$ is at most $r-h-1$. 
Indeed, by way of contradiction, suppose that there exists an $\F_q$-subspace $S$ of dimension $r-h$ in $W\cap\Gamma^\perp$. 
Then the $(t-r+h)$-dimensional $\F_{q^n}$-subspace $(S^*)^\perp$ of $\V$ contains the subspace $\Gamma$ and meets $W$ in the $(t-r+h)$-dimensional $\F_q$-subspace $S^{\perp'}$, a contradiction.
It follows that the dimension of $W+\Gamma^\perp$ is at least $t-(r-h-1)$.

Finally, if $h$ is maximal with the property of the statement, then there exists an $(h+1)$-dimensional $\F_{q^n}$-subspace $M$ of $V$ meeting $U$ in an $\F_q$-subspace of dimension at least $t-r+h+1$. 
Then $M$ corresponds to a $(t-r+h+1)$-dimensional subspace $H$ of $\V$ containing $\Gamma$ such that $\dim_q(H\cap W)\geq t-r+h+1$. 
Then $H^\perp$ is an $(r-h-1)$-dimensional subspace contained in $\Gamma^\perp$ and intersecting $W$ in the $(r-h-1)$-dimensional $\F_q$-subspace $(H\cap W)^{\perp'}$.
It follows that the dimension of $W\cap \Gamma^\perp$ is at least $r-h-1$ and hence the dimension of $W+\Gamma^\perp$ is at most $t-(r-h-1)$.
Combining with the lower bound on the dimension of $W+\Gamma^\perp$ the result follows.
	\end{proof}

	\begin{definition}
		\label{deffff}	
		\rm
		Let $U$ be a $t$-dimensional $\F_q$-subspace of $V=V(r,q^n)$. Then the $\F_q$-subspace $W+\Gamma^{\perp}$ of the quotient space $\V/\Gamma^{\perp}=V(t-r,q^n)$ will be denoted by $\bar U$ and it is the \emph{Delsarte dual} of $U$ (w.r.t.\,$\perp$).
	\end{definition}
	
	Arguing as in \cite[Remark 3.7]{CsMPZ2019} one can show that up to $\Gamma\mathrm{L}(t-r,q^n)$-equivalence, the Delsarte dual of an $\F_q$-subspace of $V$ does not depend on the choice of the non-degenerate reflexive sesquilinear forms $\beta'$ and $\beta$ on $W$ and $\mathbb{V}$, respectively.
	
	One can adapt the proof of \cite[Theorem 3.3]{CsMPZ2019} and give the following generalization.

	\begin{theorem}
		\label{thm:dual2}
		Let $U$ be a $t$-dimensional $(h,k)_q$-evasive subspace of a vector space $V=V(r,q^n)$ with $r<t$, $k< t-r+h-1$.
		Then the $(t-r+h-k-1)$-dimensional $\F_{q^n}$-subspaces meet $\bar U$ in $\F_q$-subspaces of dimension at most $t-k-2$. 
		In particular, if $\la \bar U\ra_{\F_{q^n}}$ has dimension at least $t-r+h-k-1$ then $\bar U$ is 
		a $(t-r+h-k-1,t-k-2)_q$-evasive subspace of dimension at least $t-r+h+1$ in  $\V/\Gamma^\perp=V(t-r,q^n)$.
	\end{theorem}
	\begin{proof}
		By Proposition \ref{prop:dual2}, $\bar U=W+\Gamma^{\perp}$ has dimension at least $t-r+h+1$ in $\V/\Gamma^\perp$. 
		By way of contradiction, suppose that there exists a $(t-r+h-k-1)$-dimensional $\F_{q^n}$-subspace of $\V/\Gamma^\perp$, say $M$, such that
		\begin{equation}\label{form2}
			\dim_q(M\cap \bar U)\geq t-k-1.
		\end{equation}
		Then $M=H+\Gamma^{\perp}$, for some $(t+h-k-1)$-dimensional $\F_{q^n}$-subspace $H$ of $\V$ containing $\Gamma^\perp$. For $H$, by \eqref{form2}, it follows that
			$\dim_q(H\cap W)=\dim_q(M \cap \bar U)\geq t-k-1$.
			
		Let $S$ be an $(t-k-1)$-dimensional $\F_q$-subspace of $W$ contained in $H$ and let $S^*:=\langle S\rangle_{\F_{q^n}}$. Then, $\dim_{q^n}S^*=t-k-1$,
		\begin{equation}\label{form3}
			S^{\perp'}=W\cap (S^*)^\perp \mbox{\quad and \quad}   S^{\perp'}\subset (S^*)^\perp=\langle S^{\perp'}\rangle_{\F_{q^n}}.
		\end{equation}
		Since $S\subseteq H\cap W$ and $\Gamma^\perp\subset H$, we get $S^*\subset H$ and $H^\perp\subset \Gamma$, i.e.
		\begin{equation}\label{form4}
			H^\perp\subseteq \Gamma\cap (S^*)^\perp.
		\end{equation}
		From \eqref{form4} it follows that
			$\dim_{q^n}\left(\Gamma\cap (S^*)^\perp\right)\geq \dim_{q^n}H^\perp=k-h+1$. 
		This implies that
		\[
		\dim_{q^n}\la\Gamma,(S^*)^\perp\ra_{q^n}= \dim_{q^n} \Gamma + \dim_{q^n} (S^*)^{\perp} - \dim_{q^n}\left(\Gamma\cap (S^*)^\perp\right) \leq
		t-r+h \]
		and hence $\la \Gamma, (S^*)^\perp \ra_{\F_{q^n}}$ is contained in a $(t-r+h)$-dimensional $T$ space of $\V$ containing $\Gamma$. Also, $\dim_q(S^{\perp'})=\dim_q W-\dim_q S=k+1$ and, by \eqref{form3}, we get
		\[S^{\perp'}= W\cap (S^*)^\perp\subseteq W\cap T.\]
		Then $\hat T:= T \cap V$ is an $h$-dimensional subspace of $V$ and, by recalling $U=\la W,\Gamma \ra_{\F_{q}}\cap V$,
		\[\dim_q(\hat T \cap U)=\dim_q(T \cap W)\geq \dim_q(S^{\perp'})=k+1,\]
		contradicting the fact that $U$ is $(h,k)_q$-evasive. 
	\end{proof}
	
	\begin{remark}
	\label{ddremark}
	Note that $\la \bar U\ra_{\F_{q^n}}$ has dimension at least $t-r+h-k-1$ whenever $\dim_q {\bar U} > t-k-2$, which clearly holds when $k+h+3>r$.
	\end{remark}
	
	\section{On the size of maximum evasive subspaces}
	\label{maximumsize}
	
	In this section we determine upper bounds on the dimension of maximum evasive subspaces and in some cases we show the sharpness of our results.
	
	We will need the following Singleton-like bound of Delsarte which can be proved easily by the pigeonhole principle.
	
	\begin{result}
		\label{Delsarte}
		Let $\cC$ be an additive subset of $\mathrm{Hom}_{\F_q}(U,V)$ where $\dim_q U = m$ and $\dim_q V =n$  such that non-zero maps of $\cC$ have 
		rank at least $\delta$. Then $|\cC|\leq q^{\min\{m,n\}(\max\{m,n\}-\delta+1)}$.
	\end{result}
	
	\begin{theorem}
		\label{bound}
		Let $U$ be an $(r-1,k)_q$-evasive subspace of $V=V(r,q^n)$.
		\begin{enumerate}
			\item If $k<(r-1)n$ then $\dim_q U \leq n+k-1$.
			\item If $k<r-2+n/(r-1)$ then $\dim_q U \leq n+k-r+1$.
		\end{enumerate}
	\end{theorem}
	\begin{proof}
		By definition, $\dim_{q^n}\la U\ra_{q^n}\geq r-1$. First assume $\dim_{q^n}\la U\ra_{q^n}= r-1$.
		Then $\dim_q U \leq k \leq n+k-1$ proving the first part.
		Since $k\geq r-1$, the condition of the second part implies $n > r-1$ and hence $\dim_q U \leq k < n+k-r+1$ proving the second part.
		
		Thus from now on we assume $\dim_{q^n}\la U\ra_{q^n}=r$. Fix an $\F_{q^n}$-basis in $V$ and for ${\bf x} \in V$ denote the $i$-th coordinate w.r.t. this basis by $x_i$. 
		For ${\bf a}=(a_0,\ldots,a_{r-1})\in \F_{q^n}^r$ define the $\F_q$-linear map $G_{\bf a} \colon {\bf x}\in U \mapsto \sum_{i=0}^{r-1}a_i x_i\in\F_{q^n}$ and put $\cC_U:=\{G_{\bf a} : {\bf a}\in\F_{q^n}^r\}$.
		
		Let $d$ denote the dimension of $U$ over $\F_q$.
		First we show that the non-zero maps of $\cC_U$ have rank at least $d-k$.
		Indeed, if ${\bf a}\neq\mathbf{0}$, then $\mathbf{u} \in \ker G_{\bf a}$ if and only if $\sum_{i=0}^{r-1} a_iu_i=0$, i.e. $\ker G_{\bf a}= U \cap H$, where $H$ is the hyperplane $[a_0,a_1,\ldots,a_{r-1}]$ of $V$.
		Since $U$ is $(r-1,k)_q$-evasive, it follows that $\dim \ker G_{\bf a} \leq k$ and hence the rank of $G_{\bf a}$ is at least $d-k$. Next we show that any two maps of $\cC_U$ are different.
		Suppose for the contrary $G_{\bf a}=G_{\bf b}$, then $G_{\bf a- \bf b}$ is the zero map.
		If ${\bf a - \bf b \neq 0}$, then $U$ would be contained in the hyperplane
		$[a_0-b_0,a_1-b_1,\ldots,a_{r-1}-b_{r-1}]$, a contradiction since $\la U \ra_{\F_{q^n}}=V$.
		Hence, $|\mathcal{C}_U|=q^{nr}$.
		
		The elements of $\cC_U$ form an $nr$-dimensional $\F_q$-subspace of $\mathrm{Hom}_{\F_q}(U,\F_{q^n})$ and the non-zero maps of $\cC_U$ have rank at least $d-k$. By Result \ref{Delsarte} we get	$|\cC_U|=q^{rn} \leq q^{d(n-d+k+1)}$ and hence
		\begin{equation}
			\label{rdbound}
			rn \leq d(n-d+k+1).
		\end{equation}
		
		To prove the first part, suppose for the contrary $d\geq n+k$.
		Substituting in \eqref{rdbound} gives $rn \leq n+k$ and hence $(r-1)n \leq k$, contradicting our assumption.
		
		\medskip
		
		To prove the second part, suppose for the contrary $d\geq n+k-r+2$. By \eqref{rdbound} we get $rn \leq rn+kr-r^2+2r-n-k+r-2$
		and hence $r-2+n/(r-1)\leq k$,	a contradiction.
	\end{proof}
	
Motivated by the proof of \cite[Lemma 1]{THR}, we present another bound which also implies the first bound in Theorem \ref{bound}.

\begin{theorem}
\label{double}
Let $U$ be an $(h,k)_q$-evasive subspace of $V=V(r,q^n)$. Then
\[|U|\leq \frac{(q^k-1)(q^{rn}-1)}{q^{hn}-1}+1.\]
\end{theorem}
\begin{proof}
First note that the number of $h$-dimensional $\F_{q^n}$-subspaces of $V$ is the Gaussian binomial coefficient
\[
\binom{r}{h}_{q^n}=
\frac{(q^{rn}-1)(q^{rn}-q^n)\ldots(q^{rn}-q^{(h-1)n})}{(q^{hn}-1)(q^{hn}-q^n)\ldots(q^{hn}-q^{(h-1)n})}.\]
The number of $h$-dimensional $\F_{q^n}$-subspaces of $V$ containing a fixed non-zero vector ${\bf x}$ is
\[\binom{r}{h}_{q^n}\frac{q^{hn}-1}{q^{rn}-1}.\]
Consider the pairs $({\bf x},H)$, where $H$ is an $h$-dimensional subspace of $V$ and ${\bf x}$ is a vector of $H$.
Then ${\bf 0}$ is contained in every $h$-dimensional $\mathbb{F}_{q^n}$-subspace and hence double counting the pairs above yields
\[\binom{r}{h}_{q^n}+(|U|-1)\binom{r}{h}_{q^n}\frac{q^{hn}-1}{q^{rn}-1}\leq q^k\binom{r}{h}_{q^n},\]
and the assertion follows.
\end{proof}

\begin{corollary}[First bound in Theorem \ref{bound}]
	Let $U$ be $(r-1,k)_q$-evasive with $k<(r-1)n$. Then $\dim_q U \leq n+k-1$. 
\end{corollary}
\begin{proof}
	In Theorem \ref{double} put $h=r-1$. Then
\[|U|\leq \frac{(q^k-1)(q^{rn}-1)}{q^{rn-n}-1}+1\]
follows. Since $|U|$ is a $q$-power, to prove our assertion, it is enough to show
\[\frac{(q^k-1)(q^{rn}-1)}{q^{rn-n}-1}+1< q^{n+k}.\]
After rearranging, this is equivalent to
\[q^{-n} \left(q^n-1\right) \left(q^{nr} - q^{k+n}\right)>0,\]
which clearly holds since $k<(r-1)n$.
\end{proof}

	In the next result we show that the second bound of Theorem \ref{bound} is sharp.
	
	\begin{proposition}
		\label{b1}
		If $k< r-2+\frac n {r-1}$ then in $V(r,q^n)$ there exist maximum $(r-1,k)_q$-evasive subspaces of dimension $n+k-r+1$.
	\end{proposition}
	\begin{proof}
		From the assumption on $k$, we get $k\leq n+r-1$. Let $W$ be an $n$-dimensional $(r-1,r-1)_q$-evasive subspace of $V(r,q^n)$, cf. Example \ref{Gabidulin} (note that $n\geq r$ follows from $k\geq r-1$) and $W'$ be an $\F_q$-subspace of dimension $k-r+1$ contained in a $1$-dimensional $\F_{q^n}$-subspace $\langle {\bf v}\rangle_{\F_{q^n}}$ of $V(r,q^n)$, with $\langle {\bf v}\rangle_{\F_{q^n}}\cap W=\{ {\bf 0}\}$. Then the direct sum $W\oplus W'$ is an $(r-1,k)_q$-evasive subspace of dimension $n+k-r+1$. 
	\end{proof}
	
	In the next result we show that the first bound of Theorem \ref{bound} is sharp when $k \geq (r-2)(n-1)+1$. 
	
	\begin{proposition}
		\label{ex00}
		If $k \geq (r-2)(n-1)+1$, then in $V(r,q^n)$ there exist $(r-1,k)_q$-evasive subspaces of dimension $n+k-1$.
	\end{proposition}
	\begin{proof}
		If $r=2$ and $k=n$ then $V(2,q^n)$ is $(1,n)_q$-evasive of dimension $2n$. If $r=2$ and $k<n$ then the result follows from Proposition \ref{b1}. From now on assume $r>2$.
		By Proposition \ref{onebyone} with $d=n$ and $s=n-1$, starting from a maximum $(1,1)_q$-evasive  subspace in $V(2,q^n)$, we get a $(2,n)_q$-evasive subspace of dimension $2n-1$ in $V(3,q^n)$.
		Continuing with this process it is possible to construct an $(r-1,(r-2)(n-1)+1)_q$-evasive subspace in $V(r,q^n)$ of dimension $(r-1)(n-1)+1$. By Proposition \ref{banale} the result follows for $k> (r-2)(n-1)+1$ as well.
	\end{proof}
	
	When $rn$ is even then we can prove the sharpness of the first bound of Theorem \ref{bound} also for smaller values of $k$.
	%even for smaller values of $k$. 
	
	\begin{proposition}
		\label{fromscattered}
		If there exists a $t$-dimensional scattered subspace in $V(r,q^n)$ then there exists an $(r-1,(r-1)n+1-t)_q$-evasive subspace of dimension 
		$(rn-t)$. In particular, if $rn$ is even and $k\geq rn/2-n+1$ then there exist $(r-1,k)_q$-evasive subspaces of dimension $n+k-1$.
	\end{proposition}
	\begin{proof}
		The first part follows from Proposition \ref{d2nuovo}. The second part follows from the fact that when $rn$ is even then there exist scattered subspaces of dimension $rn/2$, cf. Result \ref{result}, and from Proposition \ref{banale}.
	\end{proof}
	
	\bigskip
	
	The first case when we do not know the sharpness of Theorem \ref{bound} is when $r=3$, $n=5$ and $k=4$. In this case our bound states that the 
	dimension of an $(2,q^4)$-scattered subspace is at most $8$. The existence of such subspaces will follow from the results of Section \ref{Sec:5} and Proposition \ref{fromscattered} for infinitely many values of $q$.
	
	\bigskip
	
	To prove the next result, we follow some of the ideas of the proof of \cite[Theorem 2.3]{CsMPZ2019}.
	
	\begin{theorem}
		\label{bound1}
		Suppose that there exists a $d_2$-dimensional $(h_1-1,k_2)_q$-evasive subspace $W$ in $V(h_1,q^n)$.
		Let $U$ be a $d_1$-dimensional $(h_1,k_1)_q$-evasive subspace of $V=V(r,q^n)$ with $ d_2+h_1-k_2-1 > k_1$.
		Then
		\begin{equation}
			\label{newbound}
			d_1 \leq rn-\frac{rnh_1}{d_2}.
		\end{equation}
	\end{theorem}
	\begin{proof}
		Take $W$ in $\F_{q^n}^{h_1}=V(h_1,q^n)$, as in the assertion. 
		Let $G$ be an $\F_q$-linear transformation of $V$ with $\ker G=U$.
		Clearly, $\dim_q \mathrm{Im} \,G = rn-d_1$. 
		For each $(\mathbf{u}_1,\ldots,\mathbf{u}_{h_1}) \in V^{h_1}$ consider the $\F_{q^n}$-linear map
		\[ \tau_{\mathbf{u}_1,\ldots,\mathbf{u}_{h_1}} \colon (\lambda_1,\ldots,\lambda_{h_1}) \in W \mapsto \lambda_1 \mathbf{u}_1+\ldots+\lambda_{h_1} \mathbf{u}_{h_1}\in V.\]
		Consider the following set of $\F_q$-linear maps 
		\[ \C:=\{G \circ \tau_{\mathbf{u}_1,\ldots,\mathbf{u}_{h_1}} : (\mathbf{u}_1,\ldots,\mathbf{u}_{h_1}) \in V^{h_1} \}.\]
		Our aim is to show that these maps are pairwise distinct and hence $|\C|=q^{rnh_1}$.
		Suppose $G \circ \tau_{\mathbf{u}_1,\ldots,\mathbf{u}_{h_1}} =G \circ \tau_{\mathbf{v}_1,\ldots,\mathbf{v}_{h_1}}$. It follows that
		$ G \circ \tau_{\mathbf{u}_1-\mathbf{v}_1,\ldots,\mathbf{u}_{h_1}-\mathbf{v}_{h_1}}$
		is the zero map, i.e.
		\begin{equation}
			\label{ker}
			\lambda_1 (\mathbf{u}_1-\mathbf{v}_1)+ \ldots + \lambda_{h_1} (\mathbf{u}_{h_1}-\mathbf{v}_{h_1}) \in \ker G=U \mbox{ for each } (\lambda_1,\ldots,\lambda_{h_1}) \in W.
		\end{equation}
		
		\noindent For $i\in \{1,\ldots,h_1\}$ put ${\mathbf{z}}_i=\mathbf{u}_i-\mathbf{v}_i$, then $T:=\langle \mathbf{z}_1,\ldots,\mathbf{z}_{h_1}\rangle_{q^n}$ with $\dim_{q^n}T=t$. We want to show that $t=0$.
		
		First assume $t=h_1$. By \eqref{ker}
		\[ \{ \lambda_1\mathbf{z}_1+\ldots+\lambda_{h_1} \mathbf{z}_{h_1} : (\lambda_1,\ldots,\lambda_{h_1}) \in W \}\subseteq T\cap U, \]
		and $t=h_1$ yields $\dim_q (T \cap U) \geq \dim_q W=d_2\geq k_1+1$ (because of our assumption on $k_1$ and from $k_2 \geq h_1-1$), which is not possible since $T$ is an $h_1$-dimensional $\F_{q^n}$-subspace in $V$ and $U$ is $(h_1,k_1)_q$-evasive.
		
		Next assume $1\leq t\leq h_1-1$. W.l.o.g. we can assume $T=\langle \mathbf{z}_1,\ldots,\mathbf{z}_t\rangle_{q^n}$. Let $\Phi \colon \F_{q^n}^{h_1}\rightarrow T$ be the $\F_{q^n}$-linear map defined by the rule
		\[(\lambda_1,\ldots,\lambda_{h_1})\mapsto \lambda_1\mathbf{z}_1+\ldots+\lambda_{h_1} \mathbf{z}_{h_1}\] and let $\bar\Phi$ be the restriction of $\Phi$ on $W \leq \F_{q^n}^{h_1}$. It can be easily seen that
		\begin{equation}
			\label{eq5}
			\dim_{q^n}\ker \Phi =h_1-t
			\quad \mbox{ and } \quad \ker \bar\Phi =\ker \Phi\cap W.
		\end{equation}
		Also, by \eqref{ker}
		\begin{equation}
			\label{eq6.1}
			\mathrm{Im}\, \bar \Phi \subseteq T\cap U.
		\end{equation}
		By Proposition \ref{scendere}, $W$ is also an $(h_1-t,k_2-t+1)_q$-evasive subspace in $\F_{q^n}^{h_1}$ and hence taking \eqref{eq5} into account we get $\dim_q\ker \bar\Phi \leq k_2-t+1$, which yields
		\begin{equation}
			\label{eq7}
			\dim_q\mathrm{Im} \bar\Phi \geq d_2-(k_2-t+1).
		\end{equation} 
		By Proposition \ref{scendere}, $U$ is also a $(t,k_1-h_1+t)_q$-evasive subspace in $V$, thus by \eqref{eq6.1} we get $\dim_q \mathrm{Im} \bar\Phi \leq k_1-h_1+t$, contradicting \eqref{eq7}.
		
		Thus we proved $t=0$, i.e. $\mathbf{z}_i={\bf 0}$ for each $i\in\{1,\ldots h_1\}$ and hence $|\C|=q^{rnh_1}$.
		The trivial upper bound for the size of $\C$ is the size of $\F_q^{d_2\times(rn-d_1)}$ and the result follows.
	\end{proof}
	
	Then we obtain the following result which for $k<n$ slightly generalizes \cite[Theorem 2.3]{CsMPZ2019}.
	
	\begin{corollary}
		\label{nostrobound}
		If $k<n$ and $U$ is an $(h,k)_q$-evasive subspace in $V(r,q^n)$, then 
		\begin{equation}
			\label{knbound}
			\dim_q U \leq rn-\frac{rnh}{k+1}.
		\end{equation}
	\end{corollary}
	\begin{proof}
		Note that in $V(h,q^n)$ there exist $(h-1,h-1)_q$-evasive subspaces of dimension $k+1 \leq n$, cf. Result \ref{result}.
		Then the result follows from Theorem \ref{bound1} with $k_1=k$, $h_1=h$, $d_2=k+1$ and $k_2=h-1$.
	\end{proof}

	\section{Maximum evasive subspaces of $V(3,q^n)$, $n\leq 5$}
	\label{Sec:5}
	
	Let $U$ be a maximum $(h,k)_q$-evasive subspace of $V=V(r,q^n)$. In this section we investigate the size of $U$ for small values of $r$ and $n$. 
	First recall $h\leq k$. We will also assume $h<r$ and $k<nh$ since an $h$-dimensional $\F_{q^n}$-subspace has dimension $nh$ over $\F_q$. 
	If $k=nh$, then the whole vector space $V(r,q^n)$ is $(h,k)_q$-evasive.
	
	\medskip
	
	If $r=2$, then $h=1$ and $k<n$. By Theorem \ref{bound} $\dim_q U$ is at most $n+k-1$ and this bound is sharp, c.f. Example \ref{ex00}.
	
	\medskip
	
	From now on assume $r=3$. Then $h$ is $1$ or $2$.
	First note that when $h=k=1$, i.e. $U$ is scattered, then the bound $3n/2$ can be reached for $n$ even, cf. Result \ref{result}.
	If $h=2$ and $n \leq k < 2n$, then by Theorem \ref{bound} the dimension of $U$ is at most $n+k-1$ and by Proposition \ref{ex00} this bound can be reached. So from now on, we omit the discussion of these cases. It also means that we may assume $n\geq 3$.
	
	\medskip
	
	Recall that if $U$ has dimension $t$ over $\F_q$ then the dual of $U$ is $(r-h,(r-h)n+k-t)_q$-evasive of dimension $rn-t$ in $V(r,q^n)$.% and the Desarte dual of $U$ is $(t-r+h-k-1,t-k-2)$-evasive of dimension at least $t-r+h+1$ in $V(t-r,q^n)$. 
	
	\begin{itemize}
		
		\item Case $n=3$. 
		
		If $h=k=1$, then $U$ has dimension at most $4$ and such examples can be easily obtained by taking the $\F_q$-span of 
		a maximum scattered subspace in $L=V(2,q^3) \leq V$ and a vector ${\bf v} \in V \setminus L$.
		
		If $h=1$ and $k=2$ then Corollary \ref{nostrobound} gives $\dim_q U \leq 6$ and the dual of Example \ref{subgeom} with $m=1$ reaches this bound.
		
		If $h=2$ and $k=2$ then Corollary \ref{nostrobound} gives $\dim_q U \leq 3$ and Example \ref{subgeom} with $m=1$ reaches this bound.

		\item Case $n=4$. 
		
		If $h=1$ and $k=2$ then Corollary \ref{nostrobound} gives $\dim_q U \leq 8$, and this can be reached as the dual of 
		Example \ref{Gabidulin}.
		
		If $h=1$ and $k=3$ then Corollary \ref{nostrobound} gives $\dim_q U \leq 9$, and this can be reached starting from the previous example, cf. Proposition \ref{banale}.
		
		If $h=2$ and $k=2$ then Corollary \ref{nostrobound} gives $\dim_q U \leq 4$, and this can be reached, see Example \ref{Gabidulin}. 
		
		If $h=2$ and $k=3$ then Theorem \ref{bound} gives $\dim_q U\leq 6$, and it is reached by the maximum scattered subspaces of $V$, cf. Result \ref{result}.

		\item Case $n=5$.
		
		If $h=k=1$ then the Blokhuis--Lavrauw bound gives $\dim_q U\leq 7$.
		
		If $h=1$ and $k=2$ then Corollary \ref{nostrobound} gives $\dim_q U \leq 10$, and this can be reached as the dual of Example \ref{Gabidulin}. 
		
		If $h=1$ and $k=3$ then Corollary \ref{nostrobound} gives $\dim_q U \leq 11$, and this can be reached from the previous example, cf. Proposition \ref{banale}.
		
		If $h=1$ and $k=4$ then Corollary \ref{nostrobound} gives $\dim_q U \leq 12$, and this can be reached from the previous example, cf. Proposition \ref{banale}.
		
		If $h=2$ and $k=2$ then Corollary \ref{nostrobound} gives $\dim_q U \leq 5$, and this can be reached, cf. Example \ref{Gabidulin}.
		
		If $h=2$ and $k=3$ then the second bound of Theorem \ref{bound} gives $\dim_q U \leq 6$ and this bound is sharp, cf. Proposition \ref{b1}.
		
		If $h=2$ and $k=4$ then the first bound of Theorem \ref{bound} gives $\dim_q U \leq 8$.
	\end{itemize}
	
	We can see that the first open cases are when $n=5$ and $h=k=1$ or $h=2$ and $k=4$. 
	Note that if the bound in one of these two cases is reached then, by duality, the other bound is sharp as well.
	In the last part of this paper our aim is to find $7$-dimensional scattered subspaces of $V(3,q^5)$.
	Then the dual of such a subspace is $(2,4)_q$-evasive of dimension $8$. 
	Then by Theorem \ref{thm:dual2} and Remark \ref{ddremark} its Delsarte dual is an $8$-dimensional
	$2$-scattered $\F_q$-subspace in $V(5,q^5)$, which is maximum, by Corollary \ref{nostrobound}.
	Its dual is a $(3,9)_q$-evasive subspace of dimension $17$. 
	Here we cannot use Corollary \ref{nostrobound} since $k=9$ and $n=5$. However, we know that it is maximum by Corollary \ref{corord}. 
% 	The Delsarte dual of this subspace is $(5,6)_q$-evasive of dimension $16$ in $V(12,q^5)$.
% 	Its dual is $(7,25)_q$-evasive of dimension $44$ in $V(12,q^5)$. 
% 	This example shows that \eqref{knbound} does not hold without any assumption on $k$ and $n$ since it would give the upper bound $43$ for the dimension. 
	
	\subsection*{Maximum scattered subspaces of $V(3,q^5)$}
	
	In this section our aim is to construct scattered subspaces of dimension $7$ in $V(3,q^5)$.
	To do this, we will work in $\F_{q^{15}}$ considered as a $3$-dimensional vector space over $\F_{q^5}$. 
	Then the one-dimensional $\F_{q^5}$-subspaces can be represented as follows: $x\in \la  u \ra_{\F_{q^5}}$ if and only if 
	$x=\lambda u$ for some $\lambda \in \F_{q^5}$ and hence $x^{q^5-1}=u^{q^5-1}$. In particular, $x$ is a root of $x^{q^5}-dx$ with $d=u^{q^5-1}$. 
	We proceed with the following steps.
	\begin{enumerate}
		\item[(1)] Find $q$-polynomials $P(x)=\sum_{i=0}^7 \alpha_i x^{q^i} \in \F_{q^{15}}[x]$ with $q^7$ roots in $\F_{q^{15}}$. 
		Then $U:=\ker_q P$ is an $\F_q$-subspace of dimension $7$. 
		\item[(2)] Apply \cite[Theorem 2.1]{csajb} to obtain conditions when $P(x)$ has at most $q$ common roots with the polynomial $x^{q^5}-d x$ for each $d\in\F_{q^{15}}$ with $d^{1+q^5+q^{10}}=1$. 
		This is equivalent to ask that $\dim_q (\la u \ra_{\F_{q^5}} \cap U)\leq 1$, where $d=u^{q^5-1}$.
		\item[(3)] Show that $P(x)$ can be chosen such that these conditions are satisfied. 
	\end{enumerate}
	
	\bigskip
	
	For some $a, b \in \F_{q^{15}}$, $a^{q^3}b \neq ab^{q^3}$ put
	
	\[R_{a,b}(x)=R(x)=
	\begin{pmatrix}
	x & x^{q^3} & x^{q^6} \\
	a & a^{q^3} & a^{q^6} \\
	b & b^{q^3} & b^{q^6}
	\end{pmatrix}=x^{q^6}(ab^{q^3}-a^{q^3}b)+x^{q^3}(a^{q^6}b-ab^{q^6})+x(a^{q^3}b^{q^6}-a^{q^6}b^{q^3}).\]
	
	Then $\ker R = \la a,b\ra_{\F_{q^3}}$ has dimension $2$ over $\F_{q^3}$ and hence dimension $6$ over $\F_q$. 
	Take any $\bar{x} \notin  \ker R$ and put $c=R(\bar{x})$. Define
	\begin{equation}
		\label{Eq:P(x)}
		P_{a,b,c}(x)=P(x)=cR(x)^q-c^qR(x).
	\end{equation}
	Then $P$ vanishes on $\la \ker R, \bar{x} \ra_{\F_q}$, i.e. on an $\F_q$-subspace of dimension $7$. 
	The coefficients of $P(x)$ are:
	\[\alpha_0=-c^q(a^{q^3}b^{q^6}-a^{q^6}b^{q^3}), \quad \alpha_1=c(a^{q^4}b^{q^7}-a^{q^7}b^{q^4}),\quad \alpha_2=0,\]
	\[\alpha_3=-c^q(a^{q^6}b-ab^{q^6}),\quad \alpha_4=c(a^{q^7}b^q-a^qb^{q^7}),\quad \alpha_5=0,\]
	\[\alpha_6=-c^q(ab^{q^3}-a^{q^3}b), \quad \alpha_7=c(a^qb^{q^4}-a^{q^4}b^q).\]
	
	By \cite[Theorem 2.1]{csajb} $U$ is scattered if and only if for each $d\in \F_{q^{15}}$ with $d^{1+q^5+q^{10}}=1$ the following two matrices are not singular at the same time:
	\[\left(
	\begin{array}{cccccccccccc}
	\alpha_7^{q^4} & \alpha_6^{q^4} & 0 & \alpha_4^{q^4} & \alpha_3^{q^4} & 0 & \alpha_1^{q^4} & \alpha_0^{q^4} & 0 & 0 & 0 & 0 \\
	0 & \alpha_7^{q^3} & \alpha_6^{q^3} & 0 & \alpha_4^{q^3} & \alpha_3^{q^3} & 0 & \alpha_1^{q^3} & \alpha_0^{q^3} & 0 & 0 & 0 \\
	0 & 0 & \alpha_7^{q^2} & \alpha_6^{q^2} & 0 & \alpha_4^{q^2} & \alpha_3^{q^2} & 0 & \alpha_1^{q^2} & \alpha_0^{q^2} & 0 & 0 \\
	0 & 0 & 0 & \alpha_7^q & \alpha_6^q & 0 & \alpha_4^q & \alpha_3^q & 0 & \alpha_1^q & \alpha_0^q & 0 \\
	0 & 0 & 0 & 0 & \alpha_7 & \alpha_6 & 0 & \alpha_4 & \alpha_3 & 0 & \alpha_1 & \alpha_0 \\
	1 & 0 & 0 & 0 & 0 & -d^{q^6} & 0 & 0 & 0 & 0 & 0 & 0 \\
	0 & 1 & 0 & 0 & 0 & 0 & -d^{q^5} & 0 & 0 & 0 & 0 & 0 \\
	0 & 0 & 1 & 0 & 0 & 0 & 0 & -d^{q^4} & 0 & 0 & 0 & 0 \\
	0 & 0 & 0 & 1 & 0 & 0 & 0 & 0 & -d^{q^3} & 0 & 0 & 0 \\
	0 & 0 & 0 & 0 & 1 & 0 & 0 & 0 & 0 & -d^{q^2} & 0 & 0 \\
	0 & 0 & 0 & 0 & 0 & 1 & 0 & 0 & 0 & 0 & -d^q & 0 \\
	0 & 0 & 0 & 0 & 0 & 0 & 1 & 0 & 0 & 0 & 0 & -d \\
	\end{array}
	\right)\]

	\[\left(
	\begin{array}{cccccccccc}
	\alpha_7^{q^3} & \alpha_6^{q^3} & 0 & \alpha_4^{q^3} & \alpha_3^{q^3} & 0 & \alpha_1^{q^3} & \alpha_0^{q^3} & 0 & 0 \\
	0 & \alpha_7^{q^2} & \alpha_6^{q^2} & 0 & \alpha_4^{q^2} & \alpha_3^{q^2} & 0 & \alpha_1^{q^2} & \alpha_0^{q^2} & 0 \\
	0 & 0 & \alpha_7^q & \alpha_6^q & 0 & \alpha_4^q & \alpha_3^q & 0 & \alpha_1^q & \alpha_0^q \\
	0 & 0 & 0 & \alpha_7 & \alpha_6 & 0 & \alpha_4 & \alpha_3 & 0 & \alpha_1 \\
	1 & 0 & 0 & 0 & 0 & -d^{q^5} & 0 & 0 & 0 & 0 \\
	0 & 1 & 0 & 0 & 0 & 0 & -d^{q^4} & 0 & 0 & 0 \\
	0 & 0 & 1 & 0 & 0 & 0 & 0 & -d^{q^3} & 0 & 0 \\
	0 & 0 & 0 & 1 & 0 & 0 & 0 & 0 & -d^{q^2} & 0 \\
	0 & 0 & 0 & 0 & 1 & 0 & 0 & 0 & 0 & -d^q \\
	0 & 0 & 0 & 0 & 0 & 1 & 0 & 0 & 0 & 0 \\
	\end{array}
	\right)
	\]

	\begin{theorem}
		Let $h$ be a non-negative integer. Consider $q=p^{15h+s}$, with $\gcd(s,15)=1$ if $p=2,3$ and with $s=1$ if $p=5$.
		Then in $V(3,q^5)$ there exist scattered subspaces of dimension $7$.
	\end{theorem}
	\begin{proof}
	We will prove that there exist $a,b,c\in \mathbb{F}_{q^{15}}$ with $a^{q^3}b \neq ab^{q^3}$ and 
	$c \in \mathrm{Im}\, R_{a,b} \setminus \{0\}$ such that the kernel of $P_{a,b,c}(x)$ as in \eqref{Eq:P(x)} is scattered of dimension $7$ in $\F_{q^{15}}$.
	
	First consider the case $p=2$ and $s=1$. Choose  as a generator of $\mathbb{F}_{2^{15}}$ an element $\xi$ such that $\xi^{15} + \xi^5 + \xi^4 + \xi^2 + 1=0$.
	Then put $V=\xi^{31369}$. It can be verified with the help of the Software MAGMA \cite{MAGMA} that $\lambda^{15}  + \lambda + 1=0$. 
	Then for any element $z\in \F_{2^{15}}$ we have $z^q=z^{2^{15h+1}}=z^2$. 
	Set $a=\lambda^2$, $b=\lambda^4$ and 
	\[c:=R_{a,b}(\lambda)=\lambda^{273} + \lambda^{266} + \lambda^{161} + \lambda^{140} + \lambda^{98} + \lambda^{84}=\xi^{8539}\neq 0.\]
	Then the second matrix above, which we will call $M(d)$, reads as
	%Then the second matrix $M(d)$ from above reads as % Such a matrix reads
	$$\left(
	\begin{array}{cccccccccc}
	\xi^{3757}&   \xi^{22429}  & 0   &\xi^{17}  & \xi^{20559}  & 0 &  \xi^{10610} &  \xi^{9472} & 0 &  0\\
	0 &  \xi^{18262}  & \xi^{27598}  & 0 &  \xi^{16392} &  \xi^{26663} &  0  & \xi^{5305}   & \xi^{4736} &  0\\
	0  & 0 &  \xi^{9131}  & \xi^{13799} &  0 &  \xi^{8196}   &\xi^{29715}  & 0 &  \xi^{19036}   & \xi^{2368}\\
	0 &  0  & 0 &  \xi^{20949} &  \xi^{23283}  & 0  & \xi^{4098}  & \xi^{31241} &  0   &  \xi^{9518}\\
	1  & 0  & 0 &  0 &  0&   d^{q^5}&   0 &  0&   0  & 0\\
	0  & 1  & 0  & 0 &  0&   0&   d^{q^4} &  0&   0  & 0\\
	0  & 0 &  1  & 0 &  0&   0&   0  & d^{q^3}&   0  & 0\\
	0  & 0 &  0  & 1 &  0&   0&   0 &  0  & d^{q^2}  & 0\\
	0  & 0 &  0  & 0 &  1&   0&   0&   0  & 0   &d^q\\
	0  & 0&   0  & 0&   0&   1&   0&   0  & 0 &  0\\
	\end{array}
	\right)$$
	and its determinant is
	\begin{multline*}
		\det M(d)=\xi^{5977} d^{q^4+q^3+q^2+q} + \xi^{2592} d^{q^3+q^2+q}+ \xi^{4799} d^{q^2+q} + \xi^{8832} d^q\\ 
		+\xi^{4161} d^{q^4+q^3+q^2} + \xi^{19121} d^{q^4+q^2} + \xi^{28007} d^{q^4} + \xi^{27801}.
	\end{multline*}
	Next we prove $\det M(d)\neq 0$ for each $d\in \mathbb{F}_{q^{15}}$ with $d^{q^{10}+q^5+1}=1$. 
	To do this, we define the following multivariate polynomial over the algebraic closure $\overline{\mathbb{F}}_2$ of $\F_2$:
	\begin{multline*}
		F_0(D_0,\ldots,D_{14}):=\xi^{5977} D_4D_3D_2D_1 + \xi^{2592}D_3D_2D_1+ \xi^{4799} D_2D_1+  \xi^{8832} D_1\\+\xi^{4161} D_4D_3D_2 
		+ \xi^{19121} D_4D_2 + \xi^{28007} D_4 + \xi^{27801} \in \overline{\mathbb{F}}_{2}[D_0,\ldots,D_{14}].
	\end{multline*}
	Clearly, $\det M(d)=F_0(d,d^q,\ldots,d^{q^{14}})$. 
	For $i=1,2,\ldots,14$ define $F_i\in \overline{\mathbb{F}}_{2}[D_0,\ldots,D_{14}]$ as the polynomial obtained from $F_0$ by taking $q^i$-th powers of its coefficients and by replacing $D_j$ by $D_{j+i \pmod {15}}$ for $j=0,1,\ldots, 14$. 
	Then, by $d\in \F_{q^{15}}$, $\det M(d)=0$ yields $F_i(d,d^q,\ldots,d^{q^{14}})=0$ for $i=0,1,\ldots, 14$. 
	
	We first compute the ideal $I$ generated by the polynomials $F_i$, $i=1,\ldots,14$. 
	Then we compute  the elimination ideal of $I$ with respect to the variable $D_0$, which is generated by $G(D_0):= D_0^3 + \xi^{7925}D_0^2 + \xi^{24175}D_0 + \xi^{31682}$. The three roots of $G(D_0)$ are $\xi^{15773}, \xi^{16482}, \xi^{32194}$ and therefore  putative values  $d\in \mathbb{F}_{q^{15}}$ for which $\det M(d)=0$ belong to $S:=\{\xi^{15773}, \xi^{16482}, \xi^{32194}\}\subset \mathbb{F}_{2^{15}}$.   Recall that $d$ should also satisfy $d^{q^{10}+q^5+1}=1$ and, since $d \in S \subset \mathbb{F}_{2^{15}}$ and $q=2^{15h+1}$, this equation reads $d^{2^{10}+2^5+1}=1$. None of the  elements of $S$ is a $(2^{10}+2^5+1)$-th root of unity in $\mathbb{F}_{2^{15}}$. This shows that $\det M(d)\neq 0$ when $d^{q^{10}+q^5+1}=1$.
	
	\medskip
	
A similar argument applies to all the other cases. In particular, elements $\lambda \in \mathbb{F}{_{p^{15}}}$ for  $p=2,3$ and $\gcd(s,15)=1$ are summarized in Table \ref{Table:1}. As a notation, if  $p=2$ then $\xi$ satisfies $\xi^{15} + \xi^5 + \xi^4 + \xi^2 + 1=0$, whereas if $p=3$ then $\xi^{15} + 2\xi^8 + \xi^5 + 2\xi^2 + \xi + 1=0$ holds. Finally, for  $p=5$ and $s=1$, the element $\lambda=\xi^{24079949306}$, where  $\xi^{15} + 2\xi^5 + 3\xi^3 + 3\xi^2 + 4\xi + 3=0$, can be chosen.

	\end{proof}

\begin{remark}
A computer search in MAGMA \cite{MAGMA} shows that random $7$-dimensional $\F_q$-subspaces in $V(3,q^5)$, for $q=2,3$, are in general not scattered. Indeed, we tested in total $5180000$ and $3700$ $\F_q$-subspaces of dimension $7$ in $V(3,2^5)$ and $V(3,3^5)$, respectively, and we found that $332477$ $(6.4\%)$ and $625$ $(16.9\%)$ of them are scattered. The number of subspaces examined represents only a tiny part of the whole search space. However, our computations seem to suggest  that the property of being scattered is not so rare for $7$-dimensional $\F_q$-subspaces in $V(3,q^5)$. On the other hand, such scattered subspaces are completely random and a suitable description seems far from being reached.
\end{remark}

\section*{Acknowledgements}

The	research  was supported by the Italian National Group for Algebraic and Geometric Structures and their Applications (GNSAGA - INdAM). 
The second author was partially supported by the J\'anos Bolyai Research Scholarship of the Hungarian Academy of Sciences
and by the National Research, Development and Innovation Office – NKFIH under the grants PD 132463 and K 124950.

%\pagebreak
	
	\begin{center}
	\begin{table}[h]
	\tabcolsep= 0.7 mm
	\caption{Values $V\in \mathbb{F}{_{p^{15}}}$ for $p=2,3$ and $\gcd(s,15)=1$.}\label{Table:1}
	\begin{center}
	\begin{tabular}{|c||c|c|c|c|c|c|c|c|}
	\hline
	&$s=1$&$s=2$&$s=4$&$s=7$&$s=8$&$s=11$&$s=13$&$s=14$\\
	\hline \hline
	&&&&&&&&\\[-0.4 cm]
	$p=2$&$\xi^{31369}$&$\xi^{14336}$&$\xi^{184}$&$\xi^{136}$&$\xi^{4102}$&$\xi^{16767}$&$\xi^{28172}$&$\xi^{11248}$\\
	&&&&&&&&\\[-0.4 cm]
	\hline
	&&&&&&&&\\[-0.4 cm]
	$p=3$&$\xi^{11227515}$&$\xi^{10565258}$&$\xi^{5832991}$&$\xi^{12725576}$&$\xi^{11963627}$&$\xi^{11963627}$&$\xi^{13348604}$&$\xi^{13348604}$\\[0.1 cm]
	\hline
	\end{tabular}
	\end{center}
	\end{table}
\end{center}

\end{document}